\documentclass[12pt]{amsart}
\usepackage{latexsym}
\usepackage{amsthm}
\usepackage{amsmath}
\usepackage{amsfonts}
\usepackage{amssymb}
\usepackage[dvips]{graphicx}
\usepackage{xypic}
\input xy
\xyoption{all}

\newtheorem{theo}{Theorem}[section]

\newtheorem{coro}[theo]{Corollary}
\newtheorem{rem}[theo]{Remark}
\newtheorem{exam}[theo]{Example}

\newcommand\Set{\operatorname{\bf Set}}

\newcommand\CAT{\operatorname{\bf CAT}}
\newcommand\ACC{\operatorname{\bf ACC}}
\newcommand\Ho{\operatorname{Ho}}

\newcommand\colim{\operatorname{colim}}
\newcommand\dom{\operatorname{dom}}
\newcommand\cod{\operatorname{cod}}

\newcommand\cf{\mathcal {F}}
\newcommand\cu{\mathcal {U}}

\newcommand\ck{\mathcal {K}}
\newcommand\cl{\mathcal {L}}

\newcommand\cw{\mathcal {W}}

\date{September 21, 2011}
 
\begin{document}
\title[Colimits of accessible categories]
{Colimits of accessible categories}
\author[R. Par\' e and J. Rosick\'{y}]
{R. Par\' e and J. Rosick\'{y}$^*$}
\thanks{ $^*$ Supported by MSM 0021622409.} 
\address{
\newline 
Department of Mathematics and Statistics\newline 
Dalhousie University\newline
Halifax, NS, Canada, B3H 3J5\newline
pare@mathstat.dal.ca
\newline\newline
Department of Mathematics and Statistics\newline
Masaryk University, Faculty of Sciences\newline
Kotl\'{a}\v{r}sk\'{a} 2, 611 37 Brno, Czech Republic\newline
rosicky@math.muni.cz
}
 
\begin{abstract}
We show that any directed colimit of acessible categories and accessible full embeddings is accessible and, assuming the existence of arbitrarily 
large strongly compact cardinals, any directed colimit of acessible categories and accessible embeddings is accessible.
\end{abstract} 

\maketitle
 
\section{Introduction}
Accessible categories are closed under constructions of ``a limit type". More precisely, the 2-category of accessible categories and accessible
functors has all limits appropriate for 2-categories calculated in the 2-category of categories and functors (see \cite{MP}). The situation is much 
less satisfactory for colimits. The only general result is that lax colimits of strong diagrams of accessible categories and accessible functors 
exist and are calculated as the idempotent completion of the lax colimit of categories (see \cite{MP}, Theorem 5.4.7).  In this paper we show that 
any directed colimit of accessible categories and accessible full embeddings is accessible and, assuming the existence of a proper class of strongly 
compact cardinals, accessible categories are closed under directed colimits of embeddings. We do not know whether set theory is really necessary 
for the second result. We also do not know anything about general directed colimits. We will start with an example of a colimit of accessible categories 
which is not accessible (but has split idempotents).

All undefined concepts concerning accessible categories can be found in \cite{AR} or \cite{MP}. Recall that a functor $F:\ck\to\cl$ is called 
$\lambda$-accessible if $\ck$ and $\cl$ are $\lambda$-accessible categories and $F$ preserves $\lambda$-directed colimits. $F$ will be called 
\textit{strongly} $\lambda$-\textit{accessible} if, in addition, it preserves $\lambda$-presentable objects. Any $\lambda$-accessible functor
is strongly $\mu$-accessible for some regular cardinal $\mu$. $F$ is (strongly) accessible if it is (strongly) $\lambda$-accessible for some regular 
cardinal $\lambda$. $\CAT$ will denote the (non-legitimate) category of categories and functors while $\ACC$ is the (non-legitimate) category 
of accessible categories and accessible functors.
 
\begin{exam}\label{ex1.1}
{
\em
Let $\ck$ be a combinatorial model category and $\cw$ its class of weak equivalences. Then $\cw$ is an accessible category and its embedding 
$G:\cw\to\ck^\to$ into the category of morphisms of $\ck$ is accessible (see \cite {L} A.2.6.6 or \cite{R} 4.1). Let $\dom,\cod:\cw\to\ck$
be the functors assigning to each $f\in\cw$ its domain or codomain. These functors are accessible and let $\varphi:\dom\to\cod$ the natural 
transformation such that $\varphi_f=f$. Then the coinverter of $\varphi$ is the homotopy category $\Ho\ck=\ck[\cw^{-1}]$ of $\ck$. The homotopy 
category has very often split idempotents (for instance if $\ck$ is stable) and is almost never accessible; e.g., if $\ck$ is the model category 
of spectra then $\Ho\ck$ has split idempotents and is not accessible.
}
\end{exam}

\section{Directed colimits of accessible full embeddings}
\begin{theo}\label{th2.1} 
Let $F_{ij}: \ck_i \to\ck_j$, $i\leq j \in I$ be a directed diagram of accessible categories and accessible full embeddings. 
Then its colimit in $\CAT$ is accessible as are the colimit injections, and is in fact the colimit in $\ACC$.
\end{theo}
\begin{proof} 
We can assume the $F_{ij} $ are full inclusions. Then we want to show that $\ck=\bigcup\limits_{i\in I}\ck_i$ is accessible. Let $\kappa$ be such that 
each $\ck_i$ is $\kappa$-accessible, each inclusion $\ck_i \subseteq \ck_j $ is strongly $\kappa$-accessible , and $\kappa > |I| $.

Let $d_{mn} : K_m \to K_n$, $ m\leq n \in M$, be a $\kappa$-directed diagram in $\ck$. We claim that there is an $i_0 \in I$ and a cofinal subset 
$M_0 \subseteq M$ such that $K_m \in \ck_{i_0}$ for all $m \in M_0$. Otherwise there would be, for every $i$, an $m_i \in M$ such that 
$K_p \notin \ck_i$ for all $p\geq m_i$. Then as $M$ is $\kappa$-directed and $\kappa > |I|$, there is one $p\geq m_i$ for all $m_i$, 
and the corresponding $K_p$ is not in any $\ck_i$, a contradiction.

Now $M_0$ is $\kappa$-directed as it is cofinal in $M$ so $\colim_{m\in M_0} K_m $ exists in $\ck_{i_0}$. Now for any cocone 
$\langle k_m : K_m \to L\rangle_{m\in M}$ in $\ck$, $L$ will be in some $\ck_i$ and if we take $i_1 \geq i, i_0$ then we get a cocone 
$\langle k_m : K_m \to L\rangle_{m\in M_0}$ in $\ck_{i_1}$. As $\ck_{i_0} \subseteq \ck_{i_1} $ preserves $\kappa$-directed colimits, this cocone
factors uniquely through $\colim_{m\in M_0} K_m$, so $\colim_{m\in M_0} K_m$ is the colimit in $\ck$ as well.

If our diagram lies entirely in one $\ck_i$ to start with we can take $i_0 = i$ and $M_0 = M$, so the inclusion $\ck_i \subseteq \ck$ preserves 
$\kappa$-directed colimits.

If $K$ is $\kappa$-presentable in $\ck_i$, and $\langle K_m\rangle_{m\in M}$ is a $\kappa$-directed diagram in $\ck$, then we can choose the $i_0$ above 
so that it is also $\geq i$. Then 
\begin{align*}
\ck(\colim K_m, K)&= \ck_{i_0}(\colim K_m, K)\\
&\cong\colim \ck_{i_0}(K_m, K) = \colim \ck (K_m,K) 
\end{align*}
because $ K$ is also $\kappa$-presentable in $\ck_{i_0}$. So $K$ is $\kappa$-presentable in $\ck$ i.e. the inclusions $\ck_i \subseteq \ck$ are strongly 
$\kappa$-accessible.

Every object of $\ck$ is a $\kappa$-directed colimit of $\kappa$-presentables in some $\ck_i $ so also in $\ck$. Thus $\ck$ is $\kappa$-accessible.

Finally, $\ck$ is the colimit of $\ck_i $ in $\ACC$, for if $\langle G_i : \ck_i \to \cl\rangle_{i \in L}$ is a compatible family 
of $\kappa_i$-accessible functors, we can choose the $\kappa$ in the above argument to be larger than all $\kappa_i$, and then the extension 
$G:\ck \to \cl$ will preserve $\kappa$-directed colimits.
\end{proof}

\begin{exam}\label{ex2.2}
{
\em
Let ${\bf n} $ be the ordered set $\{1<2<\cdots<n\}$ and consider the chain of accessible embeddings
$$
\Set \to \Set^{\bf 2} \to \Set^{\bf 3} \to \cdots \to \Set^{{\bf n}} \to \cdots
$$
where the transition for ${\bf n}$ to ${\bf n+1} $ extends a path of length $n$ to one of length $n+1$ by adding an identity at the end. The colimit 
can be identified with the category of infinite paths $A_1\to A_2\to A_3\to\cdots$ which are eventually constant, i.e. there is an $N$ such that
$A_n\to A_{n+1}$ is an identity for all ${n\geq N}$. It is $\omega_1$-accessible but not $\omega$-accessible. 
}
\end{exam}

\begin{rem}\label{re2.3}
{
\em
Theorem \ref{th2.1} can be extended to directed colimits of embeddings $F_{ij}$ such that for each commutative triangle 
$$
\xymatrix@=3pc{
F_{ij}A \ar[rr]^{F_{ij}(h)}
\ar[dr]_{F_{ij}(f)} && F_{ij}C\\
& F_{ij}B \ar[ur]_{g}
}
$$
there is $\overline{g}:B\to C$ such that $F_{ij}(\overline{g})=g$.

In fact, we can choose $m_0\in M_0$ and repeat the argument above to get $i_1>i_0$ and a cofinal subset $M_1\subseteq M_0$ such that
$d_{m_0m}\in\ck_{i_1}$ for each $m\in M_1$. Since $d_{mn}\in\ck_{i_1}$ for each $m\leq n$ from $M_1$ (because $d_{mn}d_{m_0m}=d_{m_0n}$),
$\colim_{m\in M_1}K_m$ exists in $\ck_{i_1}$. Using this colimit in the proof above instead of $\colim_{m\in M_0}K_m$, we get the extension
of  \ref{th2.1}.
}
\end{rem}

\section{Directed colimits of accessible embeddings}

A cardinal $\mu$ is called \textit{strongly compact} if for every set $I$, every $\mu$-complete filter on $I$ is contained  in a $\mu$-complete
ultrafilter on $I$. Often, compact cardinals are called strongly compact. A cardinal $\mu$ is called $\alpha$-\textit{strongly compact} 
if for every set $I$, every $\mu$-complete filter on $I$ is contained  in an $\alpha$-complete ultrafilter on $I$ ($\mu$ is called
$L_{\alpha\omega}$-compact in \cite{EM}). Clearly, $\mu$ is compact if and only if it is strongly $\mu$-compact.

\begin{theo}\label{th3.1}
Let $F_{ij}:\ck_i\to\ck_j$, $i\leq j\in I$ be a directed diagram of strongly $\lambda$-accessible embeddings and $F_i:\ck_i\to\ck$ its colimit
in $\CAT$. Let $\lambda\vartriangleleft\mu$ be a strongly $\alpha$-compact cardinal where $\alpha=\max\{\lambda,|I|^+\}$.

Then $\ck$ is $\mu$-accessible and $F_i$ are strongly $\mu$-accessible.
\end{theo}
\begin{proof}
First, we will show that $\ck$ has $\mu$-directed colimits. Let $d_{mn}:K_m\to K_n$, $m\leq n\in M$ be a $\mu$-directed diagram in $\ck$. 
Let $\cf$ be the filter on $M$ generated by sets $M_m=\{k\in M | m\leq k\}$, $m\in M$. Since $M$ is $\mu$-directed, $\cf$ is $\mu$-complete 
and thus it is contained in an $\alpha$-complete ultrafilter $\cu$ on $M$. Put $M_m^i=\{k\in M_m | d_{mk}\in\ck_i\}$ for $m\in M$ and
$i\in I$. Since $M_m=\bigcup\limits_{i\in I} M_m^i$, $|I|<\alpha$ and $\cu$ is $\alpha$-complete, there is $\tilde{m}\in I$ such that 
$M_m^{\tilde{m}}\in\cu$. Using the $\alpha$-completeness of $\cu$ again, we get $U\in\cu$ such that $\tilde{m}=\tilde{n}$ for each $m,n\in U$.
We denote this common value of $\tilde{m}$ by $\tilde{U}$. Restrict our starting diagram by taking $d_{mn}:K_m\to K_n$ such that $m,n\in U$
and $d_{mn}\in\ck_{\tilde{U}}$. We get a subdiagram of the starting diagram and we will show that this subdiagram is $\lambda$-directed.
Consider a subset $X\subseteq U$ with $|X|<\lambda$. Then
$$
V=U\cap\bigcap\limits_{x\in X}M_x^{\tilde{U}}
$$
belongs to $\cu$. Thus $V\neq\emptyset$ and for $m\in V$, we have $d_{xm}\in\ck_{\tilde{U}}$ for each $x\in X$. Thus our subdiagram
is $\lambda$-directed. Let $K$ be its colimit in $\ck_{\tilde{U}}$. Since $F_{\tilde{U}}$ preserves $\lambda$-directed colimits, $K$ is a colimit 
of our subdiagram in $\ck$. For each $m\in M$, there is $n\in U\cap M_m$ because the intersection belongs to $\cu$. Thus our subdiagram
is cofinal in the whole diagram, which means that $K$ is a colimit of the starting diagram in $\ck$. 

We have proved that $\ck$ has $\mu$-directed colimits. Moreover, since any $\mu$-directed colimit in $\ck$ is calculated in some $\ck_i$, 
the embeddings $F_i$ preserve $\mu$-directed colimits. Since $\lambda\vartriangleleft\mu$, the embeddings $F_{ij}$ are strongly $\mu$-accessible. 
This means that if $A$ is $\mu$-presentable in $\ck_i$ then it is $\mu$-presentable in $\ck_j$ for all $i\leq j\in I$. 
Since any $\mu$-directed colimit in $\ck$ is calculated in some $\ck_i$, $A$ is $\mu$-presentable in $\ck$ as well. Thus the embeddings $F_i$ preserve 
$\mu$-presentable objects. Since each $K\in\ck$ belongs to some $\ck_i$ and it is a $\mu$-directed colimit of $\mu$-presentable objects in $\ck_i$, 
$K$ is a $\mu$-directed colimit of $\mu$-presentable objects in $\ck$. Thus $\ck$ is $\mu$-accessible.
\end{proof}

\begin{coro}\label{cor3.2}
Assuming the existence of arbitrarily large compact cardinals, $\ACC$ is closed in $\CAT$ under directed colimits of embbedings.
\end{coro}
\begin{proof}
Let $F_{ij}:\ck_i\to\ck_j$, $i\leq j\in I$ be a directed diagram of embeddings in $\ACC$. There is a regular cardinal $\lambda$ such that
all functors $F_{ij}$ are strongly $\lambda$-accessible (see \cite{AR} 2.19). There is a compact cardinal $\mu>\lambda,|I|^+$. Since
$\mu$ is inaccessible (see \cite{J}), \cite{MP} 2.3.4 implies that $\lambda\vartriangleleft\mu$. Following \ref{th3.1}, a colimit $\ck$
of $F_{ij}$ in $\CAT$ is $\mu$-accessible and $F_i:\ck_i\to\ck$ are $\mu$-accessible as well. Thus $F_i:\ck_i\to\ck$ is a colimit
in $\ACC$ (see the proof of \ref{th2.1}).
\end{proof}

\begin{exam}\label{ex3.3}
{
\em
Consider the following countable chain of locally finitely presentable categories and finitely accessible functors
$$
\Set \xrightarrow{\quad  F_{01}\quad} \Set^2\xrightarrow{\quad F_{12}\quad} \dots\Set^n\xrightarrow{\quad F_{nn+1}\quad}\dots 
$$
Here, $\Set$ is the category of sets, $F_{nn+1}(X_1,\dots,X_n)=(X_1,\dots,X_n,X_n)$ and $F_{nn+1}(f_1,\dots,f_n)=(f_1,\dots,f_n,f_n)$
is the action of $F_{nn+1}$ on objects and morphisms. The colimit $\Set^{<\omega}$ in $\CAT$ consists of sequences $(X_n)_{n\in\omega}$ 
which are eventually constant, i.e., there is $n\in\omega$ such that $X_n=X_m$ for all $n\leq m$. Similarly, morphisms are eventually constant 
sequences $(f_n)_{n\in\omega}$ of mappings. Following \ref{th2.1}, $\Set^{<\omega}$ is accessible assuming the existence 
of a strongly $\omega_1$-compact cardinal.

We do not know whether the accessibility of $\Set^{<\omega}$ depends on set theory.
}
\end{exam}

\end{document}